\documentclass[11pt]{article}

\usepackage{amsfonts,amsthm,amssymb,amsmath}
\usepackage{graphicx,mathrsfs}
\usepackage[usenames,dvipsnames,svgnames,table]{xcolor}
\usepackage{hyperref}
\usepackage{booktabs}

\usepackage[top=1in,bottom=1in,left=1in,right=1in,a4paper
]{geometry}
\usepackage{natbib}
\usepackage{multirow}

\usepackage{tikz}
\usetikzlibrary{matrix}
\usetikzlibrary{calc,shapes,arrows}
\usepackage{pgfplots}
\usetikzlibrary{patterns}

\numberwithin{equation}{section}
\numberwithin{table}{section}
\numberwithin{figure}{section}

\newtheorem{definition}{Definition}[section]
\newtheorem{assumption}{Assumption}[section]
\newtheorem{theorem}{Theorem}[section]
\newtheorem{lemma}{Lemma}[section]
\newtheorem{proposition}{Proposition}[section]

\newcommand{\brackets}[1]{\left({#1}\right)}

\newcommand{\R}{\mathbb{R}} 
\newcommand{\N}{\mathbb{N}} 

\newcommand{\C}{\mathbf{C}} 

\newcommand{\fmm}[1]{\bar{\mathcal{#1}}}
\newcommand{\fm}[1]{\bar{#1}}

\newcommand{\tfm}[1]{\tilde{#1}}
\newcommand{\tfmm}[1]{\tilde{\mathcal{#1}}}

\usepackage[normalem]{ulem}

\usepackage[normalem]{ulem}

\begin{document}

\title{
A Note on Many-server Fluid Models with Time-varying Arrivals}
\author{
 Zhenghua Long \ and \ Jiheng Zhang\\
 The Hong Kong University of Science and Technology\\
 \{\href{mailto:zlong@connect.ust.hk}{zlong@conncet.ust.hk},
 \href{mailto:jiheng@ust.hk}{jiheng@ust.hk}\}
}
\date{}

\maketitle

\begin{abstract}
We extend the measure-valued fluid model, which tracks residuals of patience and service times, to allow for time-varying arrivals.
The fluid model can be characterized by a one-dimensional convolution equation involving both the patience and service time distributions.
We also make an interesting connection to the measure-valued fluid model tracking the elapsed waiting and service times.
Our analysis shows that the two fluid models are actually characterized by the same one-dimensional convolution equation.
\end{abstract}

\emph{Key words: many-server queue, fluid model, time-varying, abandonment}

\section{Introduction}

There has been increasing interest in developing and analyzing fluid models of many-server queues with general service and patience time distributions since the pioneering work \cite{Whitt2006}. As an example of how powerful the fluid model approach is that it can be used to approximate a system with dependent service and patience times, see  \cite{BassambooRandhawa2016,WBP2018}.
The research community has developed measure-valued processes and two-parameter processes to describe the system dynamics due to the generality of the distributions.
Existing studies can be divided into two categories. The first  tracks the elapsed waiting and service times of all customers in the system, see \cite{Whitt2006} and \cite{KangRamanan2010}. The second  tracks the residual patience and service times, see \cite{Zhang2013}.

The first line of works is represented by \cite{KangRamanan2010}, which is based on \cite{KaspiRamanan2011} on the model without abandonment.
\cite{KangRamanan2010} requires rather complicated conditions on the hazard rate of the distributions (see Assumption~3.3 in \cite{KangRamanan2010}). \cite{Zuniga2014} extends \cite{KangRamanan2010} by relaxing their assumptions.
However, both in \cite{KangRamanan2010} and \cite{Zuniga2014}, the existence of a solution to the fluid model is proved using stochastic approximation.

The fluid model tracking elapsed times is also developed in \cite{LiuWhitt2011a,LiuWhitt2012}, which adapt the approach in  \cite{Whitt2006} to allow the number of servers and service/patience time distributions to vary with time. 
Moreover, they provide a direct analysis on the fluid model tracking elapsed times to obtain existence and uniqueness by assuming two key assumptions: (i) the system alternates between overloaded and underloaded intervals, and (ii) the functions specifying the fluid model are suitably smooth.
The direct analysis on the fluid model tracking elapsed times is also studied in \cite{Kang2014}, which assumes that the service time distribution has a density and the hazard rate function of the patience time distribution is locally bounded.

In the second line of works tracking residual times, \cite{Zhang2013} directly proves the existence and uniqueness of the many-server fluid model with a constant arrival rate only requiring continuity of the service time distribution and Lipschitz continuity of the patience time distribution. Moreover, it builds the foundation to prove the convergence to the equilibrium state in \cite{LongZhang2014}.
However, the modeling approach in \cite{Zhang2013} seems a bit inflexible as extending the analysis of the fluid model with a constant arrival rate to time-varying arrival rates is not that straightforward.
Another downside of \cite{Zhang2013} is the condition on initial state of the queue, which assumes that initial customers are those who arrived in the past following an arrival process with the same arrival rate.

This paper extends the measure-valued fluid model tracking residual times in \cite{Zhang2013}, where a fluid model is studied for the $G/GI/n+GI$ queue, to allow for time-varying arrivals. In this paper, we  focus on the study of the fluid model of many-server
queues with time-varying arrival rates, and general service and patience
time distributions. The queueing model is denoted by $G_t/GI/n+GI$. The $G_t$
represents a general time-varying arrival process. The first $GI$ indicates
that service times are independent and identically distributed (i.i.d.) with a general distribution.
The $n$ denotes the number of homogeneous servers.  There is an unlimited
waiting space, called the buffer, where customers wait to be served according
to the first-come-first-served (FCFS) discipline. Customers are only allowed
to abandon if their patience times expire before their service starts. Again,
the patience times are i.i.d. and with a general distribution (the second
$GI$).  

We also provide a unified approach to study the two types of fluid models tracking elapsed times and residual times. We show that both types of fluid models are characterized by a same convolution equation, which is proved to  possess a unique solution. Thus, both types of fluid models are alternative to approximate the original stochastic processes. We  address the following open issues regarding the
fluid model.
\begin{enumerate}
\item Can we extend the measure-valued fluid model of \cite{Zhang2013}, which tracks residuals, to allow for time-varying arrivals?
\item What is the fundamental mathematical law driving the dynamics of both types of  fluid models (tracking elapsed and residual times)?
\end{enumerate}

We aim to address these two questions by first extending \cite{Zhang2013} to allow for time-varying arrivals in Section~\ref{sec:FM-residual}.
We derive a one-dimensional convolution equation \eqref{eq:key-equation-time-varying} as the key insight of the fluid model in Section~\ref{sec:analysis}, where several properties of the fluid model are also developed.
Analysis of the key equation \eqref{eq:key-equation-time-varying} is presented in Section~\ref{sec:key-equation}.
We identify a connection between the initial conditions required by both fluid models when analyzing the fluid model tracking elapsed times in Section~\ref{sec:FM-elapsed}.
We show that \eqref{eq:key-equation-time-varying} also serves as the  foundation of the fluid model tracking elapsed times.

\section{The Fluid Model Tracking Residual Times}
\label{sec:FM-residual}

Let $\R$ denote the set of real numbers and $\R_+=[0,\infty)$. For $a,b\in\R$, write $a^+$ for the positive part of $a$ and $a\wedge b$ for the minimum. For convenience of notation, define $C_x=(x,\infty)$. We append a bar sign on processes to indicate that they are fluid model processes and to be consistent with the notations in \cite{Zhang2013}.

We consider a fluid model of the $G_t/GI/n+GI$
queue  with time-varying arrival process
\begin{equation}
\label{def:Et-fluid}
  \fm{E}(t)=\int_0^t\lambda(s)ds,
  \quad \lambda(\cdot)\ge 0.
\end{equation}
For $t<0$, let $\lambda(t)$ be the arrival rate of the fluid arriving before time 0.
Following the modeling approach in \cite{Zhang2013}, we introduce the \emph{virtual buffer} which holds all the fluid that has not yet scheduled to enter service even when their patience is exhausted. 
When the fluid is admitted to service, the system will check whether the fluid has positive remaining patience time or not. Only the fluid with positive remaining patience time will enter service, otherwise it will abandon the system.
Thus, the fluid in the virtual buffer is allowed to have negative remaining patience time. 
For any time $t\in[0,\infty)$, let $\fmm{R}(t)(C_x)$ denote the amount of fluid in the  virtual buffer with remaining patience time larger than $x\in\R$; and $\fmm{Z}(t)(C_x)$ denote the amount of fluid in service with remaining service time larger than $x\ge 0$. We assume customers' patience times and service times are mutually independent and follow the distributions $F$ and $G$, respectively. See \cite{BassambooRandhawa2016,WBP2018} for the study of dependent  service and patience time distributions.

Denote by $\fm{R}(t)$, $\fm{Q}(t)$ and $\fm{Z}(t)$ the amount of
 fluid in the virtual buffer, in the queue and in service at time $t$, respectively. Then they can be recovered from $\fmm{R}$ and $\fmm{Z}$ as follows
\begin{align} 
   \label{eq:def-R-Q-Z-Zhang}
  \fm{R}(t)=\fmm{R}(\R), \quad
  \fm{Q}(t)=\fmm{R}(t)(C_0) \quad \text{and}\quad
  \fm{Z}(t)=\fmm{Z}(t)(C_0),
\end{align}
where $C_0=(0,\infty)$ since the fluid in the queue or in service cannot have 0 remaining times.
Let $\fm{X}(t)=\fm{Q}(t)+\fm{Z}(t)$ denote the total fluid content in the system.
We assume that the initial fluid arrives at some negative time $t\in(-\infty,0)$.
So our arrival process $\fm{E}(t)$ extends to the negative axis and we introduce $\omega(t)$ as the solution to
\begin{equation*}
  \fm{R}(t)
  = \int_{t-\omega(t)}^t d\fm{E}(s).
\end{equation*}
Intuitively,  $\omega(t)$ can be considered
as the  waiting time of the earliest arrived fluid in the virtual buffer
at time $t$. And $t-\omega(t)$ can be thought of as the arrival time of the earliest arrived fluid in the virtual buffer at time $t$.
We also introduce
\begin{equation*}
  \fm{B}(t)=\fm{E}(t)-\fm{R}(t),
\end{equation*}
and $\fm{B}(t)-\fm{B}(s)$ can be regarded as the fluid content {leaving} the virtual buffer during the time interval $(s,t]$.
Note that the processes $\fm{R}$, $\fm{Q}$, $\fm{Z}$ and $\fm{X}$ are all derived directly from the measure-valued process $(\fmm{R},\fmm{Z})$, and the processes $\omega$ and $\fm{B}$ are derived by combining the measure-valued process and the arrival process. The fluid model is defined as follows.

\begin{definition}[The fluid model tracking residual]
  \label{def:fluid-models-tracking-residual}
  The fluid model $\{\fmm{R}(t),\fmm{Z}(t)\}$ satisfies the dynamic equations
  \begin{align}
    \label{def:fmmRt-Zhang}
    \fmm{R}(t)(C_x)
    &=\int_{t-\omega(t)}^t F^c(x+t-s)d\fm{E}(s),\quad x\in\R,\\
    \label{def:fmmZt-Zhang}
    \fmm{Z}(t)(C_x)
    &=\fmm{Z}(0)(C_{x+t})
    +\int_0^t F^c\left( \omega(s)\right)G^c(x+t-s)d\fm{B}(s),\quad x\in\R_+,
\end{align}
and the non-idling constraints
\begin{align}
    \label{nonidlingQ-Zhang}
    \fm{Q}(t) &= (\fm{X}(t)-1)^+,\\
    \label{nonidlingZ-Zhang}
    \fm{Z}(t) &= \fm{X}(t)\wedge 1.
\end{align}
Moreover, the initial state $(\fmm{R}(0),\fmm{Z}(0))$ satisfying \eqref{def:fmmRt-Zhang} and \eqref{def:fmmZt-Zhang} at time $t=0$ has no atoms.
\end{definition}

The intuition behind the above definition resembles that of (3.1)--(3.2) in \cite{Zhang2013}. The difference is that $\omega(t)$ simply reduces to $\fm{R}(t)/\lambda$ when the arrival rate is constant and equals $\lambda$.
We want to emphasize here that the dynamic equations \eqref{def:fmmRt-Zhang}--\eqref{def:fmmZt-Zhang} implicitly assume the FCFS policy. 
In general, only specifying the remaining patience times in {the} queue does not give a full picture of the status of the queue. 
For example, assuming there are two customers with remaining patience times 1 and 10 in {the} queue, the measure does not tell us who is the first in {the} queue. To overcome this issue, we incorporate FCFS into the dynamic equations. For any $s\in[t-\omega(t),t]$, among the infinitesimal amount of arriving fluid $d\fm{E}(s)$, the fraction of remaining patience time larger than $x$ at time $t$ is $F^c(x+t-s)d\fm{E}(s)$ as shown in \eqref{def:fmmRt-Zhang}.
By the definition of $\fm{B}$ and $\omega$, it is easy to see that
\begin{align}
    \label{eq:Bt-Et-omega-t}
    \fm{B}(t) = \fm{E}(t-\omega(t)).
\end{align}
The infinitesimal amount of fluid $d\fm{B}(t)$ that is about to enter service at time $t$ actually arrived at time $t-\omega(t)$.  Only a fraction $F^c(\omega(t))$, with the original patience time larger than the waiting time $\omega(t)$, actually enters service. This is characterized by \eqref{def:fmmZt-Zhang}.

For direct analysis of the fluid model, we need the following assumption on  the service and patience time distributions throughout this paper.

\begin{assumption}
  \label{assump:GF}
  The service time distribution $G$ is continuous with finite mean $1/\mu$, and the patience time distribution $F$ is Lipschitz continuous.
\end{assumption}

\subsection{Properties and Analysis}
\label{sec:analysis}

\paragraph{Preliminary analysis.}
We first perform some preliminary analysis to arrive at the key equation \eqref{eq:key-equation-time-varying}.
It follows from \eqref{def:fmmRt-Zhang} that
\begin{align}
    \fm{Q}(t)&=\fmm{R}(t)(C_0)\nonumber\\
    \label{eq:def-Qt-omegat}
             &=\int_{t-\omega(t)}^{t}F^c(t-s)\lambda(s)ds
              =\int_0^{\omega(t)} F^c(s)\lambda(t-s)ds.
\end{align}
For any $t\ge 0$, introduce two new functions
\begin{align}
  \label{eq:def-Ft-f-lambda-t-s}
  F_t(x) &= \int_0^x f(s)\lambda(t-s)ds, \\
  \label{eq:def-Fdt}
  F_{d,t}(x) &= \int_0^x F^c(s)\lambda(t-s)ds,
\end{align}
where $f(x)=(d/dx)F(x)$ exists since every Lipschitz continuous function is absolutely continuous (Page~112 in \cite{Royden1988}).
The domain for both functions is $x\in[0,t+\omega(0)]$ since the fluid model at $t$ only depends on the arrival process from time $-\omega(0)$ to time $t$. For any $t\geq 0$, denote by $N_{F,t}$ the maximum value of $F_{d,t}(\cdot)$; \begin{align}
    \label{eq:def-N-Ft}
    N_{F,t}=F_{d,t}(t+\omega(0)).
\end{align}
Using \eqref{eq:def-Fdt}, \eqref{eq:def-Qt-omegat} becomes
\begin{align}
    \label{eq:Qt-Fdt}
    \fm{Q}(t)=F_{d,t}(\omega(t)).
\end{align}
It follows from \eqref{def:fmmZt-Zhang} and \eqref{eq:Bt-Et-omega-t} that
\begin{align*}
  \fm{Z}(t) &= \fmm{Z}(t)(C_0)\nonumber\\
  &= \fmm{Z}(0)(C_t)+\int_0^t F^c(\omega(s))G^c(t-s)d\fm{E}(s-\omega(s)).
\end{align*}
Since a monotone function is of bounded variation, it  follows from Lemma~\ref{lem:at-nonnegative} and \eqref{nonidlingQ-Zhang} that $\fm{Q}(t)$ is also of bounded variation. Thus, applying the chain rule to \eqref{eq:def-Qt-omegat} gives
\begin{align}
  \label{eq:derivative-qt-second-equation}
  d\fm{Q}(t) &=\lambda(t)dt
               -F^c(\omega(t))\lambda(t-\omega(t))d(t-\omega(t))
               -\int_{t-\omega(t)}^t f(t-s)\lambda(s)ds \nonumber\\
             &=\lambda(t)dt
               -F^c(\omega(t))d\fm{E}(t-\omega(t))
               -F_t(\omega(t)) dt,
\end{align}
so that
\begin{align*}
  \fm{Z}(t) = \fmm{Z}(0)(C_t)
              +\int_0^t G^c(t-s)\left[ \lambda(s)-F_s(\omega(s)) \right]ds
              -\int_0^t G^c(t-s)d\fm{Q}(s).
\end{align*}
Performing change of variable and integration by parts, we have
\begin{align}
  \label{eq:pre-key-equation}
  \begin{split}
    \fm{Z}(t) &= \fmm{Z}(0)(C_t)
    + \frac{1}{\mu}\int_0^t \lambda(t-s)-F_{t-s}(\omega(t-s))dG_e(s)\\
    &\quad - \fm{Q}(t) + \fm{Q}(0)G^c(t)+\int_0^t \fm{Q}(t-s)dG(s),
  \end{split}
\end{align}
where $G_e(\cdot)$ is the equilibrium distribution associated with $G$ defined as
\begin{align}
\label{eq:Geequilibrium}
  G_e(x)=\mu\int_0^x G^c(y)dy.
\end{align}
Based on \eqref{eq:def-Ft-f-lambda-t-s} and \eqref{eq:def-Fdt}, we introduce the following function for all $t\geq 0$,
\begin{align}
  \label{eq:def-Ht}
  H_t(y) =
  \begin{cases}
    \lambda(t)-F_t(F_{d,t}^{-1}(y)),  &\text{if }0\leq y< N_{F,t},\\
    \lambda(t)-F_t(F_{d,t}^{-1}(N_{F,t})),  &\text{if }y\geq N_{F,t},
  \end{cases}
\end{align}
where  $N_{F,t}$ is defined in \eqref{eq:def-N-Ft} and $F_{d,t}^{-1}(y)=\inf\{x\geq 0: F_{d,t}(x)\geq y\}$ for all $y\in[0,N_{F,t}]$.
{By \eqref{nonidlingQ-Zhang} and \eqref{eq:Qt-Fdt}, $\omega(t)=F_{d,t}^{-1}((\fm{X}(t)-1)^+)$. Combining this with \eqref{nonidlingQ-Zhang}, \eqref{eq:pre-key-equation} and \eqref{eq:def-Ht}},  we obtain the following \emph{key equation}
\begin{align}
  \label{eq:key-equation-time-varying}
  \begin{split}
    \fm{X}(t)
    &= \fmm{Z}(0)(C_t)+\fm{Q}(0)G^c(t)\\
    &\quad +\frac{1}{\mu}\int_0^t H_{t-s}((\fm{X}(t-s)-1)^+)dG_e(s)
    +\int_0^t(\fm{X}(t-s)-1)^+dG(s).
  \end{split}
\end{align}

\paragraph{Existence and uniqueness of a solution to the fluid model.}
It follows from the proof of Theorem~3.1 in \cite{Zhang2013} that there is a one-to-one correspondence between the measure-valued process $(\fmm{R},\fmm{Z})$ and the one-dimensional process $\fm{X}$. Thus, the existence and uniqueness of a solution to the fluid model in Definition~\ref{def:fluid-models-tracking-residual} is equivalent to the existence and uniqueness of the solution to the key equation \eqref{eq:key-equation-time-varying}, which is proved in Proposition~\ref{pro:key-equation} in Section~\ref{sec:key-equation}. Below we immediately have the following theorem.

\begin{theorem}[Existence and uniqueness]
  \label{thm:E_and_U}
  Under Assumption~\ref{assump:GF}, there exists a unique solution to the measure-valued fluid model $\{\fmm{R}(t),\fmm{Z}(t)\}$ in \eqref{def:fmmRt-Zhang}--\eqref{nonidlingZ-Zhang}.
\end{theorem}

\paragraph{Time shift of the fluid model.}
For any $\tau\geq 0$, denote $(\fmm{R}_\tau(t),\fmm{Z}_{\tau}(t))=(\fmm{R}(\tau+t),\fmm{Z}(\tau+t))$.
The time shift for all the derived ``status'' quantities such as $\omega_\tau(\cdot)$, $\fm{R}_\tau(\cdot)$,
$\fm{Q}_\tau(\cdot)$, $\fm{Z}_\tau(\cdot)$ and $\fm{X}_\tau (\cdot)$ is defined in the same way, e.g., $\omega_\tau(t)=\omega(\tau+t)$. However, the time shift for the ``cumulative'' process $\fm{E}_\tau(t)$ is defined as $\fm{E}_\tau(t)=\fm{E}({\tau+t})-\fm{E}(\tau)$ (similarly for $\fm{B}_\tau(\cdot)$).
If we think of the arrival rate, then $\fm{E}'_\tau(s)=\lambda_{\tau}(s)=\lambda(\tau+s)$.
The following proposition shows that the fluid model can be ``restarted'' at time $\tau>0$ by viewing $(\fmm{R}(\tau),\fmm{Z}(\tau))$ as the initial condition.

\begin{proposition}[Time-shifted fluid model]
    \label{lem:time-shift-fluid-model}
    The time-shifted fluid solution $(\fmm{R}_\tau(t),\fmm{Z}_\tau(t))$ satisfies
\begin{align}
    \label{def:fmmRt-Zhang-shift}
    &\fmm{R}_\tau(t)(C_x)
    =\int_{t-\omega_\tau(t)}^t F^c(x+t-s)d\fm{E}_\tau(s),\quad x\in\R,\\
    \label{def:fmmZt-Zhang-shift}
    &\fmm{Z}_\tau(t)(C_x)
    =\fmm{Z}(\tau)(C_{x+t})
     +\int_0^t F^c\left( \omega_\tau(s)\right)G^c(x+t-s)d\fm{B}_\tau(s),\quad x\in\R_+.
\end{align}
And the shifted key equation becomes
\begin{align}
\label{eq:shifted-key-equation-tau}
  \begin{split}
  \fm{X}_\tau(t)&=\fmm{Z}(\tau)(C_t)+\fm{Q}(\tau)G^c(t)\\
  &\quad+\frac{1}{\mu}\int_0^t H_{\tau+t-s}((\fm{X}_\tau(t-s)-1)^+)dG_e(s)
  +\int_0^t (\fm{X}_{\tau}(t-s)-1)^+dG(s).
  \end{split}
\end{align}
\end{proposition}
\begin{proof}
Replacing $t$ in \eqref{def:fmmRt-Zhang} by $\tau+t$ gives
\begin{align*}
    \fmm{R}_{\tau}(t)(C_x) &= \fmm{R}(\tau+t)(C_x)\\
    &= \int_{\tau+t-\omega(\tau+t)}^{\tau+t}
       F^c(x+\tau+t-s)d\fm{E}(s)\\
    &= \int_{t-\omega(\tau+t)}^t F^c(x+t-s)d\fm{E}(\tau+s),
\end{align*}
where the last equation is due to change of variable. This implies \eqref{def:fmmRt-Zhang-shift} by using the definition of the time shift.
Similarly, replacing $t$ by $\tau+t$ in \eqref{def:fmmZt-Zhang} yields
\begin{align*}
  \fmm{Z}_\tau(t)(C_x) &= \fmm{Z}(\tau+t)(C_x)\\
  &= \fmm{Z}(0)(C_{x+\tau+t}) +\int_0^{\tau+t} F^c(\omega(s))G^c(x+\tau+t-s)d\fm{B}(s)\\
  &= \fmm{Z}(0)(C_{x+\tau+t})
     +\int_0^{\tau}F^c(\omega(s))G^c(x+\tau+t-s)d\fm{B}(s)\\
     &\quad +\int_\tau^{\tau+t} F^c(\omega(s))G^c(x+\tau+t-s)d\fm{B}(s)\\
  &= \fmm{Z}(\tau)(C_{x+t}) +\int_0^t F^c(\omega(\tau+s))G^c(x+t-s)d\fm{B}(\tau+s),
\end{align*}
which implies \eqref{def:fmmZt-Zhang-shift} by the definition of the time shift.

Replacing $(t,x)$ in \eqref{def:fmmZt-Zhang} by $(\tau,t)$ yields
\begin{align*}
    \fmm{Z}(\tau)(C_t)=\fmm{Z}(0)(C_{\tau+t})
    +\int_0^\tau F^c\left( \omega(s)\right)G^c(\tau+t-s)d\fm{B}(s).
\end{align*}
Combining the above with \eqref{eq:Bt-Et-omega-t}, \eqref{eq:derivative-qt-second-equation} and \eqref{eq:def-Ht}, we can verify that
\begin{align*}
    \fmm{Z}(\tau)(C_t)&=\fmm{Z}(0)(C_{\tau+t})
           +\frac{1}{\mu}\int_t^{\tau+t} H_{\tau+t-s}(\fm{Q}(\tau+t-s)dG_e(s)\\
         &\quad -\fm{Q}(\tau)G^c(t)+\fm{Q}(0) G^c(\tau+t)
                +\int_t^{\tau+t}\fm{Q}(\tau+t-s)dG(s).
\end{align*}
Thus, the right-hand side of \eqref{eq:shifted-key-equation-tau} becomes
\begin{align*}
    &\fmm{Z}(0)(C_{\tau+t})+\fm{Q}(0)G^c(\tau+t)
    +\frac{1}{\mu}\int_0^{\tau+t}H_{\tau+t-s}(\fm{Q}(\tau+t-s))dG_e(s)\\
    &+\int_{0}^{\tau+t}\fm{Q}(\tau+t-s) dG(s),
\end{align*}
which equals $\fm{X}(\tau+t)$ by \eqref{eq:key-equation-time-varying}. Thus \eqref{eq:shifted-key-equation-tau} follows by applying the time-shift definition.
\end{proof}

\paragraph{Special case with a constant arrival rate.}
We specialize the time-varying arrival rate to be constant, i.e., $\lambda(\cdot)\equiv\lambda$.
It can be seen from Lemma~\ref{lem:queue-upbound-lemma} that any solution to \eqref{eq:key-equation-time-varying} satisfies
\begin{align*}
  (\fm{X}(t)-1)^+ \leq
   \lambda\int _0^{t+\omega(0)} F^c(s)ds
  \quad\text{for all }t\geq 0.
\end{align*}
It follows from \eqref{eq:def-Ht} that for any $t\geq 0$,  $H_t(y)=\lambda
H(y)$ for all $y\in [0, \lambda\int_0^{t+\omega(0)} F^c(s)ds]$, where
\begin{align*}
  H(y) &=
  \begin{cases}
    F^c\brackets{F_d^{-1}\brackets{\frac{y}{\lambda}}}, & \text{if }0\leq y
<\lambda N_F,\\
    0, & \text{if } y\geq \lambda N_F,
  \end{cases}
\end{align*}
with $F_d(x)=\int_0^x F^c(s)ds$ and $N_F=\int_0^\infty F^c(s)ds$.
Thus we can replace $H_{t-s}(\cdot)$ in \eqref{eq:key-equation-time-varying}
by $\lambda H(\cdot)$ and obtain the following key equation for this special case:
\begin{equation*}
  \fm{X}(t) = \fmm{Z}(0)(C_t)+\fm{Q}(0)G^c(t)
  +\frac{\lambda}{\mu}\int_0^t H((\fm{X}(t-s)-1)^+)dG_e(s)
  +\int_0^t(\fm{X}(t-s)-1)^+dG(s),
\end{equation*}
which is consistent with the key equation (4.6) in \cite{Zhang2013}.

\paragraph{Balance equations} Regarding the last term in \eqref{def:fmmZt-Zhang}, we introduce an auxiliary process
\begin{align*}
  \fm{A}(t) = \int_0^t F^c(\omega(s))d\fm{B}(s),
\end{align*}
which can be interpreted as the {amount of fluid that} actually enters service.
By \eqref{eq:Bt-Et-omega-t}, \eqref{eq:Qt-Fdt}, \eqref{eq:derivative-qt-second-equation} and \eqref{eq:def-Ht}, the auxiliary process can be written as
\begin{align*}
  \fm{A}(t) = \int_0^t H_s(\fm{Q}(s))ds-\fm{Q}(t)+\fm{Q}(0).
\end{align*}
Denote by $\fm{L}(t)$ the \emph{abandonment} process, which can be derived from the following balance equation of the physical queue,
\begin{align*}
  \fm{Q}(t) = \fm{Q}(0)+\fm{E}(t)-\fm{L}(t)-\fm{A}(t).
\end{align*}
From the above two equations, we get
\begin{align*}
  \fm{L}(t)
            = \fm{E}(t)-\int_0^t H_s(\fm{Q}(s))ds.
\end{align*}
Using \eqref{eq:def-Ft-f-lambda-t-s}, \eqref{eq:Qt-Fdt} and \eqref{eq:def-Ht} yields
\begin{align*}
  \fm{L}(t) = \int_0^t\int_0^{\omega(s)} f(x)\lambda(s-x)dxds.
\end{align*}

According to the fluid dynamic equation \eqref{def:fmmZt-Zhang},
\begin{align*}
  \fm{Z}(t) = \text{
                     {$\fmm{Z}(0)(C_t)$}}
              +\int_0^t G^c(t-s)d\fm{A}(s).
\end{align*}
Then the \emph{service completion} process, denoted by $\fm{S}(t)$, can be derived from the following balance equation of the server pool,
\begin{align*}
  \fm{Z}(t) = \fm{Z}(0)+\fm{A}(t)-\fm{S}(t).
\end{align*}
That is
\begin{align*}
  \fm{S}(t) = \fmm{Z}(0)((0,t])
              +\int_0^t G(t-s)d\fm{A}(s).
\end{align*}
It is clear that the balance equation of the fluid content in the
system satisfies
\begin{align*}
    \fm{X}(t)=\fm{X}(0)+\fm{E}(t)-\fm{L}(t)-\fm{S}(t).
\end{align*}

 Note that the introduced processes $\fm{A}$, $\fm{L}$, $\fm{S}$ and the balance equations are not needed in the definition and analysis of the fluid model. We only provide them here for completeness and potential future use.

\section{The One-dimensional Convolution Equation}
\label{sec:key-equation}

We analyze the key equation \eqref{eq:key-equation-time-varying} in this section.  
Denote by {$\C[0,\infty)$} the space of continuous functions on the interval {$[0,\infty)$}. The following Proposition~\ref{pro:key-equation} showing the existence and uniqueness of the solution to \eqref{eq:key-equation-time-varying} is the main result of this paper. The auxiliary Lemmas~\ref{lem:at-nonnegative}--\ref{le:conse-tau}, which also reveal some additional properties of the solution to \eqref{eq:key-equation-time-varying}, are placed in the appendix.
\begin{proposition}
  \label{pro:key-equation}
  There exists a unique solution {$\fm{X}\in{\C}[0,\infty)$}
to \eqref{eq:key-equation-time-varying}.
\end{proposition}
\begin{proof}
We first prove that there exists a number $b>0$ and a unique continuous function $\fm{X}(t)$ satisfies \eqref{eq:key-equation-time-varying} when $0\le t\le b$. And then we extend the solution indefinitely. According to the value of $\fm{X}(0)$, we consider the following two cases.

\noindent
{\textbf{Case 1}}: $\fm{X}(0)\leq 1$. This implies $\omega(0)=0$ by \eqref{eq:def-Qt-omegat}. 
Deduce from Lemma~\ref{lem:queue-upbound-lemma} that for all $t\geq 0$,
\begin{align}
    \label{eq:pro_queue_upperbound_initial}
    (\fm{X}(t)-1)^+
    \leq \int_0^t F^c(s)
    \lambda(t-s)ds
    =N_{F,t}.
\end{align}
Let $M$ be any strictly positive number and $S_F=\inf\{x\geq 0: F(x)=1\}$.
 From \eqref{eq:def-Ht} the following derivative is bounded for all $t\in[0,\frac{S_F\wedge M}{2}]$:
\begin{align*}
  \frac{d}{dy}{H}_{t}(y)
  = -\frac{f(F_{d,t}^{-1}(y))}{F^c(F_{d,t}^{-1}(y))}
  \geq -\frac{L_F}{F^c(t)}
  \geq - \frac{L_F}{F^c(\frac{S_F\wedge M}{2})},
  \quad \text{if }0\leq y \le N_{F,t},
\end{align*}
where $L_F$ is denoted to be the Lipschitz constant of $F$ by Assumption~\ref{assump:GF}. So we can pick $b_1=\frac{S_F\wedge M}{2}$ and then for any $t\in[0,b_1]$ the function ${H}_{t}(\cdot)$ in \eqref{eq:key-equation-time-varying} is Lipschitz continuous. Let $L=\frac{L_F}{F^c(\frac{S_F\wedge M}{2})}$ be the Lipschitz constant.
By Assumption~\ref{assump:GF}, there exists a $b_2>0$ such that
\begin{align*}
  \kappa := 
  \frac{1}{\mu}L[G_e(b_2)-G_e(0)]+[G(b_2)-G(0)]<1.
\end{align*}
Let $b=\min\{b_1,b_2\}$.
For any $x\in{C}[0,b]$, define 
\begin{align*}
  \Psi(x)(t) &= \fmm{Z}(0)(C_t)+\fm{Q}(0)G^c(t)
             +\frac{1}{\mu}\int_0^t \text{{${H}_{t-s}$}}((x(t-s)-1)^+)dG_e(s)
               +\int_0^t(x(t-s)-1)^+dG(s).
\end{align*}
It is clear that $\Psi(x)(t)$ is continuous in $t$, so $\Psi$ is a mapping from $\C[0,b]$ to $\C[0,b]$.
Let $\rho(x,x')=\sup_{t\in[0,b]}|x(t)-x'(t)|$ denote the uniform distance between two functions in $\C[0,b]$.
For any $x,x'\in C[0,b]$, we have
\begin{align*}
    \rho(\Psi(x),\Psi(x'))
    &\leq \sup_{t\in[0,b]} \frac{1}{\mu}\int_0^t
          L|(x(t-s)-1)^+-(x'(t-s)-1)^+|dG_e(s)\\
    &\quad + \sup_{t\in[0,b]}\int_0^t
           |(x(t-s)-1)^+-(x'(t-s)-1)^+| dG(s)\\
    &\leq \frac{1}{\mu} L\int_0^{b} \rho(x,x')dG_e(s)
          +\int_0^{b}\rho(x,x')dG(s)\\
    &\leq \kappa \rho(x,x').
\end{align*}
Since $\kappa<1$, $\Psi$ is a contraction mapping on $\C[0,b]$ under the uniform topology $\rho$.
Note that $\C[0,b]$ is complete under the uniform topology of $\rho$ (cf.~p.~80 in \cite{Billingsley1999}).
Thus, by the contraction mapping theorem (e.g., Theorem~3.2 in \cite{HunterNach2001}), $\Psi$ has a unique fixed point $x$, i.e., $x=\Psi(x)$.
This proves {that \eqref{eq:key-equation-time-varying}} has a unique solution on $[0,b]$ in this case. 

\noindent
{\textbf{Case~2}}: $\fm{X}(0)>1$. Due to the continuity of the solution to \eqref{eq:key-equation-time-varying} (if there is any) proved in Lemma~\ref{lem:xt-continuous-function}, there exists $b_3>0$ such that
\begin{align}
    \label{eq:proof-Xt-1}
    \fm{X}(t)\geq  1 \quad\text{for all }t\in[0,b_3].
\end{align}
For notational simplicity, denote $q(t)=(\fm{X}(t)-1)^+$ and
\begin{align}
    \label{eq:case2-proof-non-decreasing_initial}
    a(t)=\int_0^t H_{s}(q(s))ds-q(t)+q(0).
\end{align}
For $t\in[0,b_3]$, \eqref{eq:proof-At-renewal-non-decreasing} obtained in the proof of Lemma~\ref{lem:at-nonnegative} becomes
\begin{align*}
    {a}(t)=1-\fmm{Z}(0)(C_t)+\int_0^t {a}(t-s)dG(s).
\end{align*}
Let $G^{n*}$ be the $n$-fold convolution of $G$ with itself, and denote $U_G(t)=\sum_{i=0}^\infty G^{n*}$. The solution to the above renewal equation is
\begin{align*}
  a(t)
  = \int_0^t {(}
  1-\fmm{Z}(0)(C_{t-s})
  {)}
  dU_G(s),
  \quad t\in[0,b_3].
\end{align*}
It is clear that $a(t)$ is continuous.
Since $H_{t}(\cdot)$ is continuous, with a known ${a}(t)$ there exists a continuous solution $q(t)$ to the equation \eqref{eq:case2-proof-non-decreasing_initial} following from Theorem~II.1.1 in \cite{Miller1971}. 

Next we prove the uniqueness. Assume that ${q}_1(t)$ and ${q}_2(t)$ satisfy \eqref{eq:case2-proof-non-decreasing_initial} on the interval $[0,b_3]$.  Let
\begin{align*}
    \mathcal{L}(t):=(q_1(t)-q_2(t))^2,
    \quad t\in[0,b_3].
\end{align*}
Then, on the interval $[0,b_3]$ we can see from \eqref{eq:case2-proof-non-decreasing_initial}
that
\begin{align*}
    \mathcal{L}'(t)=2[q_1(t)-q_2(t)]
                     [H_{t}(q_1(t))-H_{t}(q_2(t))]
                  \leq 0,
\end{align*}
where the last inequality is due to the fact that $H_{t}(\cdot)$ is non-increasing from \eqref{eq:def-Ht}. Thus
$\mathcal{L}(t)$ is non-increasing on $[0,b_3]$. Since $\mathcal{L}(0)=0$ and
$\mathcal{L}(t)\geq 0$, $\mathcal{L}(t)=0$ for all $t\in[0,b_3]$.
Hence
\begin{align*}
    q_1(t)=q_2(t)\quad\text{for all }t\in[0,b_3].
\end{align*}
Thus  \eqref{eq:key-equation-time-varying} only has one
solution on the interval $[0,b_3]$. By Corollary II.2.6 in \cite{Miller1971}, we can further extend the solution to a point $\tau>0$, where $\fm{X}(\tau)=1$.  If there is no such a finite
time point, the existence and uniqueness immediately follow. Otherwise, starting from $\tau$, we can apply a similar argument as the above Case~1 to extend the solution to an extra interval with length $b$. Since the argument involves the time-shifted fluid model equation \eqref{eq:shifted-key-equation-tau}, we provide a rigorous proof in Lemma~\ref{le:conse-tau}.  As a result, we can at least get the unique solution of \eqref{eq:key-equation-time-varying} on the interval $[0,b]$ in this case.

Combing the above two cases yields that there exists a unique continuous function $\fm{X}(t)$ satisfying \eqref{eq:key-equation-time-varying} when $0\le t\le b$. Here, the definition of $b$ is same as the one in Lemma~\ref{le:conse-tau}. Thus, applying Lemma~\ref{le:conse-tau} consecutively at $\tau=b, 2b, \cdots$,
we can extend the existence
and uniqueness to $[kb,(k+1)b]$, $k=1,2,\cdots$ to the whole interval $[0,\infty)$,
proving the result.
\end{proof}

\section{The Fluid Model Tracking Elapsed Times}
\label{sec:FM-elapsed}

We now present the fluid model tracking elapsed times following earlier works in this direction, e.g., \cite{KangRamanan2010} and \cite{LiuWhitt2012}. 
Let $\fmm{R}_a(t)([0,x])$ denote the amount of fluid in the \emph{potential queue} with elapsed waiting time no larger than $x$. A potential queue holds all the fluid that has arrived but has not yet abandoned, no matter whether it has entered service or not. 
Note that the virtual buffer is employed in the fluid model tracking the residual times, and the potential queue is used in the fluid model tracking elapsed times. 
Let ${\fmm{Z}}_a(t)([0,x])$ denote the amount of fluid in the server pool with elapsed service time no larger than $x$. 
The head count processes of fluid amount in the potential queue, in the queue and in service can be recovered from $\fmm{R}_a$ and $\fmm{Z}_a$ as follows
\begin{align*}
    \fm{R}_a(t)=\fmm{R}_a(t)([0,\infty)),\quad
    \fm{Q}(t)=\fmm{R}_a(t)([0,\omega(t)])\quad \text{and}\quad
    \fm{Z}(t)=\fmm{Z}_a(t)([0,\infty)),
\end{align*}
where, as in Section~\ref{sec:FM-residual}, $\omega(t)$ represents the waiting time of the earliest arrived fluid content in the physical queue. 
For convenience of notations, let $r(t,x)$ and $z(t,x)$ be the densities of the measures $\fmm{R}_a(t)$ and $\fmm{Z}_a(t)$, respectively. In details, $r(t,x)=(d/dx)\fmm{R}_a(t)([0,x])$ and $z(t,x)=(d/dx){\fmm{Z}}_a(t)([0,x])$, which exist almost everywhere since $\fmm{R}_a(t)([0,x])$ and $\fmm{Z}_a(t)([0,x])$ are non-decreasing in $x$ (see \cite{Royden1988}, Page~100).
We have the following definition for the fluid model tracking elapsed times.

\begin{definition}[The fluid model tracking elapsed times]
  \label{def:FM-elapsed}
  The fluid model $\{\fmm{R}_a(t),\fmm{Z}_a(t)\}$ satisfies the following dynamic equations
  \begin{align}
    \label{measurevaluedQ}
    \fmm{R}_a(t)([0,x]) &= \int_0^{(x-t)^+}\frac{F^c(s+t)}{F^c(s)}r(0,s)ds
                           +\int_{(t-x)^+}^t F^c(t-s)d\fm{E}(s), \quad x\in\R_+,\\
    \label{measurevaluedB}
    \fmm{Z}_a(t)([0,x]) &= \int_{0}^{(x-t)^+}\frac{G^c(s+t)}{G^c(s)}z(0,s)ds
                           +\int_{(t-x)^+}^tG^c(t-s)d\fm{A}(s), \quad x\in\R_+,
  \end{align}
  where $\fm{A}(s)$ is the {amount of fluid that} enters service {by time $s$}.
  Moreover,  the abandonment process, denoted by $\fm{L}(t)$,
satisfies  \begin{align}
    \label{def:abandon-process-WKK}
    \fm{L}(t) = \int_0^t\left( \int_0^{\omega(s)}\frac{f(x)}{1-F(x)}r(s,x)dx \right)ds,
  \end{align}
  where $f$ is the density function of $F$.
  The fluid model needs to satisfy the balance equation
  \begin{align}
    \label{eq:balance-Qt-WKK}
    \fm{Q}(t)=\fm{Q}(0)+\fm{E}(t)-\fm{L}(t)-\fm{A}(t),
  \end{align}
  and the non-idling constraints \eqref{nonidlingQ-Zhang}--\eqref{nonidlingZ-Zhang}.
\end{definition}

Note that when $x\geq t$, the fluid content in the potential queue with elapsed waiting time less than or equal to $x$ consists of two parts: initial fluid in the queue with age $s\in[0,(x-t)^+]$ at time $0$ and fluid that arrived during $[(t-x)^+,t]$. 
For the initial fluid content, only a fraction $F^c(s+t)/F^c(s)$ of the infinitesimal amount of fluid ${r}(0,s)ds$ would still be in the potential queue at time $t$. For the fluid that  arrived at time $s\in[(t-x)^+,t]$, a proportion $F^c(t-s)$ of the infinitesimal amount $d\fm{E}(s)$ will not reach its patience time at time $t$. 
When $x<t$, the fluid content in the potential queue with elapsed waiting time less than or equal to $x$ only consists the fluid arriving at $s\in[(t-x)^+,t]$ and the explanation is exactly the same. 
The explanations for $\tfmm{Z}(t,x)$ and \eqref{def:abandon-process-WKK} are similar. We refer to \cite{KangRamanan2010} and \cite{LiuWhitt2012} for more detailed discussions on the intuition behind this definition. 

It is worth pointing out that the waiting time $\omega(t)$, abandonment process $\fm{L}(t)$ and the balance equation \eqref{eq:balance-Qt-WKK} are needed in Definition~\ref{def:FM-elapsed}, while they are derived from the model defined in Definition~\ref{def:fluid-models-tracking-residual}. The reason is that the same measure $(\fmm{R}_a(t),\fmm{Z}_a(t))$ at time $t$ could represent two different states if we are given two different $\omega(t)$. However, for the fluid model in Section~\ref{sec:FM-residual}, we can uniquely determine $\omega(t)$ once the measure $\fmm{R}(t)$ is given. Of course, this hinges on the validity of the initial condition. In the following, we show a connection between the initial conditions in both types of fluid models.

\paragraph{Correspondence between initial conditions}
Given any initial state $(r(0,\cdot),z(0,\cdot))$ in the fluid model tracking elapsed times, we can construct a corresponding initial state in the fluid model tracking residual times. 
Let
\begin{align}
    \fmm{R}(0)(C_x)
    &:=\int_{-\omega(0)}^0 F^c(x-s)\lambda(s)ds ,\quad x\in\R,\nonumber\\
\label{eq:correspondence-initial-condition-Z}
    \fmm{Z}(0)(C_x)
    &:=\int_{0}^{\infty}\frac{G^c(s+x)}{G^c(s)}z(0,s)ds,\quad x\in\R_+,
\end{align}
where $\lambda(s)$, $s\in (-\omega(0),0)$, is set to be
\begin{align}
    \label{eq:correspondence-initial-condition}
    \lambda(s)=\frac{r(0,-s)}{F^c(-s)}
\end{align}
and can be regarded as the arrival rate of the fluid arriving before time 0.
It is easy to see that the initial state $(\fmm{R}(0),\fmm{Z}(0))$ satisfies \eqref{def:fmmRt-Zhang}
and \eqref{def:fmmZt-Zhang} at time $t=0$. 

\begin{proposition}[Identical key equation]
    The measure-valued fluid model $\{\fmm{R}_a(t),\fmm{Z}_a(t)\}$ that tracks elapsed times is characterized via the same key equation  \eqref{eq:key-equation-time-varying}.
\end{proposition}
\begin{proof} 
It follows from \eqref{eq:correspondence-initial-condition} that for all $x>t$ the measure-valued process \eqref{measurevaluedQ} becomes
\begin{align}
    \fmm{R}_a(t)([0,x])&=\int_0^{x-t}\frac{F^c(s+t)}{F^c(s)}{r}(0,s)ds
                  +\int_{0}^t F^c(t-s)d\fm{E}(s)\nonumber\\
                &=\int_0^{x-t} F^c(s+t)\lambda(-s)ds
                  +\int_0^t F^c(t-s)d\fm{E}(s)\nonumber\\
                  \label{measurevaluedQ-lambda}
                &=\int_{t-x}^{t} F^c(t-s)d \fm{E}(s).
\end{align}
It can also be seen from \eqref{measurevaluedQ} that \eqref{measurevaluedQ-lambda}  still holds for all $x\leq t$.
From the above we have
\begin{align}
    \fm{Q}(t)&=\fmm{R}_a(t)([0,\omega(t)])\nonumber\\
    \label{eq:qt-age-derive}
            &=\int_{t-\omega(t)}^t F^c(t-s)d\fm{E}(s)
             =\int_0^{\omega(t)}F^c(s)\lambda(t-s)ds\nonumber\\
            &=F_{d,t}(\omega(t)),
\end{align}
where the last equation follows from the definition of $F_{d,t}(\cdot)$ in \eqref{eq:def-Fdt}.
From \eqref{measurevaluedQ-lambda}, we obtain $r(t,x)=F^c(x)\lambda(t-x)$. Thus \eqref{def:abandon-process-WKK} becomes
\begin{align*}
    \fm{L}(t)&=\int_0^t\int_0^{\omega(s)}
              f(x)\lambda(s-x)dxds\\
             &=\fm{E}(t)-\int_0^t H_s(\fm{Q}(s))ds,
\end{align*}
where the last equation follows from $H_s(\cdot)$ in \eqref{eq:def-Ht} and \eqref{eq:qt-age-derive}. Combining this with \eqref{eq:balance-Qt-WKK} yields
\begin{align*}
    \fm{A}(t)=\int_0^t H_s(\fm{Q}(s))ds-\fm{Q}(t)+\fm{Q}(0).
\end{align*}
Plugging $x=\infty$ and the above into \eqref{measurevaluedB} then combining with \eqref{eq:correspondence-initial-condition-Z}, we have
\begin{align*}
    \fm{Z}(t)=\fmm{Z}(0)(C_t)+\int_0^t G^c(t-s)H_s(\fm{Q}(s))ds
               -\int_0^t G^c(t-s)d\fm{Q}(s),
\end{align*}
Performing change of variable and integration by parts, we have
\begin{align*} 
    \fm{X}(t)
             &=\fmm{Z}(0)(C_t)+\fm{Q}(0)G^c(t)
              +\frac{1}{\mu}\int_0^t H_{t-s}((\fm{X}(t-s)-1)^+)dG_e(s)
              +\int_0^t(\fm{X}(t-s)-1)^+dG(s),
\end{align*}
which is exactly the same as the key equation in \eqref{eq:key-equation-time-varying}.
\end{proof}   

\appendix
\section{Auxiliary Lemmas}

\begin{lemma}
    \label{lem:at-nonnegative}
If there is any function $\fm{X}(t)$ satisfying \eqref{eq:key-equation-time-varying}, then
\begin{align*}
    \int_0^t H_s((\fm{X}(s)-1)^+)ds-(\fm{X}(t)-1)^+
    +(\fm{X}(0)-1)^+
\end{align*}
is non-decreasing.
\end{lemma}
\begin{proof} 
To simplify the notation, let ${q}(t)=(\fm{X}(t)-1)^+$ and
\begin{align}
    \label{eq:proof-non-decreasing}
    {a}(t)=\int_0^t H_s(q(s))ds
           -q(t)+q(0).
\end{align}
So we just need to prove that $a(\cdot)$ is non-decreasing. By \eqref{eq:Geequilibrium} and \eqref{eq:key-equation-time-varying},
\begin{align*}
  \fm{X}(t)&=\fmm{Z}(0)(C_t)+q(0)G^c(t)
  +\int_0^t H_{t-s}(q(t-s))[1-G(s)]ds
  +\int_0^t q(t-s)dG(s)\\
  &=\fmm{Z}(0)(C_t)+q(0)G^c(t)
  +\int_0^t H_{s}(q(s))ds
  -\int_0^t H_{s}(q(s))G(t-s)ds
  +\int_0^t q(t-s)dG(s).
\end{align*}
The second last term on the above equation satisfies
\begin{align*}
  \int_0^t H_{s}(q(s))G(t-s)ds
  &=\int_0^t \int_0^{t-s}
   H_{s}(q(s))dG(x)ds\\
  &=
  \int_0^t\int_0^{t-x} H_{s}(q(s))ds dG(x),
\end{align*}
where the last equality follows by changing the order of integration. So
we obtain
\begin{align*}
  \fm{X}(t)
  =\fmm{Z}(0)(C_t)+q(0)
  +\int_0^t H_{s}(q(s))ds
  -\int_0^t\left[\int_0^{t-x} H_{s}(q(s))ds
  -q (t-x)+q(0)\right] dG(x).
\end{align*}
According to the above definition of $a(t)$, we have
\begin{align}
    \label{eq:proof-At-renewal-non-decreasing}
    {a}(t)=(\fm{X}(t)\wedge 1)
           -\fmm{Z}(0)(C_t)
           +\int_0^t a(t-s)dG(s).
\end{align}
We now use \eqref{eq:proof-non-decreasing} and \eqref{eq:proof-At-renewal-non-decreasing} to show that $a(\cdot)$ is non-decreasing. Choose $b>0$ such that $G(b)<1$. We first show that $a(\cdot)$ is non-decreasing on the interval $[0,b]$. Let
\begin{align*}
    a^*=\inf_{0\leq t\leq t'\leq b}a(t')-a(t).
\end{align*}
We will prove by contradiction that $a^*\geq 0$, which implies that $a(\cdot)$ is non-decreasing on $[0,b]$. Assume {to} the contrary that $a^*<0$. Choose any $t_1,t_2\in[0,b]$ with $t_1\leq t_2$ and consider the following two cases.\\
\textbf{Case~1:} If $\fm{X}(t_2)\geq 1$, then $\fm{X}(t_2)\wedge 1=1$.
Applying \eqref{eq:proof-At-renewal-non-decreasing}, we have
\begin{align*}
\begin{split}
    a(t_2)-a(t_1)&=(\fm{X}(t_2)\wedge 1)-(\fm{X}({t_1})\wedge 1)
                  -\fmm{Z}(0)(C_{t_2})+\fmm{Z}(0)(C_{t_1})\\
                 &\quad+ \int_{t_1}^{t_2}a(t_2-s)-a(0)dG(s)
                       +\int_0^{t_1}a(t_2-s)-a(t_1-s)dG(s),
\end{split}
\end{align*}
where $a(0)=0$ from \eqref{eq:proof-non-decreasing}.
 So due to the fact $\fmm{Z}(0)(C_{t})$ is non-increasing that
\begin{align*}
  a(t_2)-a(t_1)\geq \int_0^{t_2} a^*dG(s)
  = a^* G(t_2)
  \geq a^* G(b),
\end{align*}
where the last inequality follows from the assumption that $a^*$ is negative.\\
\textbf{Case~2:} If $\fm{X}(t_2)<1$. Let $\tau=\sup\{s<t_2:\fm{X}(s)\geq 1\}\vee 0$ be the last
time that $\fm{X}$ is larger than or equal to $1$. Thus $\fm{X}(t)<1$ for all
$t\in(\tau,t_2]$. 
Then by \eqref{eq:proof-non-decreasing} and \eqref{eq:def-Ht},
\begin{align}
  \label{eq:lemma-case2-at2-at}
  a(t_2)-a(t)
  =\int_t^{t_2}H_s(0)ds
  +q(t)
  \geq \int_t^{t_2}\lambda(s)ds
  \quad\text{for all }
  t\in[\tau,t_2].
\end{align}
If $t_1\in[\tau,t_2]$, then from the above we have $a(t_2)-a(t_1)\geq 0\ge a^* G(b)$. If $t_1\in [0,\tau)$, then it is only possible when $\tau>0$.
If $\fm{X}(\tau)\geq 1$, we can  apply the same analysis in the above case (where $\fm{X}(t_2)\geq 1$) at time $\tau$ to obtain $a(\tau)-a(t_1)\geq a^* G(b)$. This together with \eqref{eq:lemma-case2-at2-at} shows that $a(t_2)-a(t_1)\geq a^* G(b)$. Otherwise, if $\fm{X}(\tau)< 1$, from the definition of $\tau$ we can find a sequence $\tau_n\in (t_1,\tau)$ satisfying $\tau_n\to \tau$ as $n$ goes to infinity and $\fm{X}(\tau_n)\geq 1$ for all $n\in\N$.  Applying case~1 at each time epoch $\tau_n$ obtains $a(\tau_n)-a(t_1)\geq a^*G(b)$. Combining this with \eqref{eq:proof-non-decreasing} and \eqref{eq:lemma-case2-at2-at}  yields
\begin{align*}
  a(t_2)-a(t_1)
  &=a(t_2)-a(\tau)
   +a(\tau)-a(\tau_n)+a(\tau_n)-a(t_1)\\
  &\geq 
    0 
    + \int_{\tau_n}^\tau H_s(q(s))ds -q(\tau)+q(\tau_n)
    +a^*G(b).
\end{align*}
 Note that $q(\tau)=0$ since we have $\fm{X}(\tau)< 1$. 
Thus, the above inequality also yields $a(t_2)-a(t_1)\geq a^*G(b)$ since $q(\tau_n)\ge 0$ and  $\lim\limits_{n\to\infty} \tau_n =\tau$. 
Summarizing both cases of $\fm{X}(t_2)$, we have
\begin{align*}
    a(t_2)-a(t_1)\geq a^* G(b).
\end{align*}
Taking infimum over $0\leq t_1\leq t_2\leq b$ gives $a^*\geq a^* G(b)$. Since $G(b)<1$, it contradicts the assumption $a^*<0$. So we must have $a^*\geq 0$, which implies that $a(t)$ is non-decreasing on $[0,b]$.

We next extend the monotonicity to $\R_+$ proving by induction. Suppose we can show that $a(\cdot)$ is non-decreasing on the interval $[0,nb]$ for some $n\in\N$. Let
\begin{align}
    \label{eq:shifted-Z-nb}
    \fmm{Z}(nb)(C_t)=\fmm{Z}(0)(C_{nb+t})
              +\int_0^{nb}G^c(nb+t-s)d a(s).
\end{align}
It is clear that the shifted fluid versions of \eqref{eq:proof-non-decreasing} and \eqref{eq:proof-At-renewal-non-decreasing} satisfy
\begin{align*}
     {a}_{nb}(t)&=\int_0^t H_{nb+s}(q_{nb}(s))ds
           -q_{nb}(t)+q_{nb}(0),\\
    {a}_{nb}(t)&=(\fm{X}_{nb}(t)\wedge 1)
           -\fmm{Z}(nb)(C_t)
           +\int_0^t a_{nb}(t-s)dG(s).
\end{align*}
To show that $a(\cdot)$ is non-decreasing on $[nb,(n+1)b]$ is the same as showing that $a_{nb}(\cdot)$ is non-decreasing on $[0,b]$. For this purpose, it is enough to verify that $\fmm{Z}({nb})(C_t)$ is non-increasing. This is obviously true due to the fact that $a(\cdot)$ is non-decreasing on $[0,nb]$ and by the definition of $\fmm{Z}_{nb}(C_t)$ in \eqref{eq:shifted-Z-nb}. Thus we extend the non-decreasing interval to $[0,(n+1)b]$. By induction, the function $a(\cdot)$ is non-decreasing on the whole interval $[0,\infty)$.
\end{proof}   
{\begin{lemma}
  \label{lem:xt-continuous-function}
  If there is any function
$\fm{X}(t)$ satisfying \eqref{eq:key-equation-time-varying}, then $\fm{X}(t)$
is a continuous function, i.e., $\fm{X}(t)\in \mathbf{C}[0,\infty)$.
\end{lemma}
\begin{proof}
  Let us denote the non-decreasing formula in Lemma~\ref{lem:at-nonnegative} by
\begin{align*}
  a(t)=\int_0^t H_s((\fm{X}(s)-1)^+)ds
  -(\fm{X}(t)-1)^+ +(\fm{X}(0)-1)^+.
\end{align*}
Then we can transform \eqref{eq:key-equation-time-varying} to be
\begin{align*}
  \fm{X}(t)
  =\fmm{Z}(0)(C_t)+(\fm{X}(0)-1)^+
  +\int_0^t H_s((\fm{X}(s)-1)^+)ds
  -\int_0^t G(t-s)d a(s).
\end{align*}
It suffices to prove the continuity of $\int_0^t G(t-s)d a(s)$. For any $0\le t_1<t_2$, we can see from the monotonicity of $a(t)$ that
\begin{align*}
 0&\le \int_0^{t_2} G(t_2-s)da(s)-
  \int_0^{t_1} G(t_1-s)da(s)\\
  &=
  \int_{t_1}^{t_2} G(t_2-s)da(s)
  +\int_0^{t_1} [G(t_2-s)-G(t_1-s)]da(s).
\end{align*}
Obviously, the right hand side of the above equality could be arbitrarily small as long as $t_1$ and $t_2$ are close enough.
Thus, the result holds.
\end{proof}}

{\begin{lemma}
  \label{lem:queue-upbound-lemma}
  If there is any function
$\fm{X}(t)$ satisfying \eqref{eq:key-equation-time-varying}, then 
\begin{align}
  \label{eq:queue-upbound-lemma}
  (\fm{X}(t)-1)^+
  \leq N_{F,t}
  =\int_0^{t+\omega(0)}F^c(s)\lambda(t-s)ds
  \quad\text{for all }
  t\geq 0,
\end{align}
where $N_{F,t}$ is denoted in \eqref{eq:def-N-Ft}.
\end{lemma}
\begin{proof}
  For notational simplicity, let $q(t)=(\fm{X}(t)-1)^+$. By \eqref{eq:def-R-Q-Z-Zhang} and \eqref{def:fmmRt-Zhang}, the initial state satisfies $q(0)=\int_0^{\omega(0)}F^c(s)\lambda(-s)ds$. This implies that \eqref{eq:queue-upbound-lemma} holds at $t=0$. Suppose there exists $t_1>0$ such that $q(t_1)>N_{F,t_1}$. Let $t_0=\sup\{s<t_1:q(s)\leq N_{F,s}\}$. Then due to the continuity proven in Lemma~\ref{lem:xt-continuous-function}, we have $q(t)-N_{F,t}\geq 0$ for all $t\in[t_0,t_1]$. By Lemma~\ref{lem:at-nonnegative} and \eqref{eq:def-Ht},
\begin{align}
  q(t_1)-q(t_0)
  &\leq \int_{t_0}^{t_1} H_s(q(s))ds\nonumber\\
    \label{eq:lemma2-qt1-qt0}
  &=\int_{t_0}^{t_1}\lambda(s)ds
  -\int_{t_0}^{t_1}\int_0^{s+\omega(0)}
  f(x)\lambda(s-x)dxds.
\end{align}
Apply change of variable to the last term 
\begin{align*}
  &\quad\int_{t_0}^{t_1}\int_{-\omega(0)}^s
  f(s-x)\lambda(x)dxds\\
  &=\int_{-\omega(0)}^{t_0}dx
    \int_{t_0}^{t_1}f(s-x)\lambda(x)ds
    +\int_{t_0}^{t_1}dx\int_x^{t_1} f(s-x)\lambda(x)ds\\
  &=\int_{t_0}^{t_1}\lambda(x)dx
    -\int_0^{t_1+\omega(0)} F^c(x)\lambda(t_1-x)dx
    +\int_0^{t_0+\omega(0)}F^c(x)\lambda(t_0-x)dx,
\end{align*}
where the first equality follows by changing the order of integration.
Plugging the above into \eqref{eq:lemma2-qt1-qt0} yields
\begin{align*}
  q(t_1)\leq q(t_0)
  +\int_0^{t_1+\omega(0)}F^c(x)\lambda(t_1-x)dx
  -\int_0^{t_0+\omega(0)}F^c(x)\lambda(t_0-x)dx.
\end{align*}
Then by the definition of $t_0$, the above implies $q(t_1)\leq \int_0^{t_1+\omega(0)}F^c(x)\lambda(t_1-x)dx=N_{F,t_1}$. This is a contradiction. 
So \eqref{eq:queue-upbound-lemma} follows.
\end{proof}}
\begin{lemma}
    \label{lem:t-omegat-nondecreasing}
If there is any function $\fm{X}(t)$ satisfying \eqref{eq:key-equation-time-varying}, then 
\begin{align}
    \label{eq:lemma3-arrival-time-non-decreasing}
    t-F_{d,t}^{-1}((\fm{X}(t)-1)^+)
\end{align}
is non-decreasing.
\end{lemma}
\begin{proof} 
As in the proof of Lemmas~\ref{lem:at-nonnegative}--\ref{lem:queue-upbound-lemma}, we also denote $q(t)=(\fm{X}(t)-1)^+$ and 
\begin{align}
    \label{eq:lemma3-at-notational-simplification}
    a(t)=\int_0^t H_s(q(s))ds-q(t)+q(0).
\end{align}
Meanwhile, let $\varpi(t)=F_{d,t}^{-1}(q(t))$ to simplify the notation. Then by \eqref{eq:def-Fdt} {and \eqref{eq:queue-upbound-lemma}} we obtain
\begin{align}
    \label{eq:qt-definition-lemma3}
    q(t)=\int_0^{\varpi(t)} F^c(s)\lambda(t-s)ds
    =\int_{t-\varpi(t)}^{t} F^c(t-s)\lambda(s)ds.
\end{align}
Applying the chain rule to the above equation yields
\begin{align*}
    dq(t) &\text{{$=\lambda(t)dt
               -F^c(\varpi(t))\lambda(t-\varpi(t))
                d(t-\varpi(t))
            -\int_{t-\varpi(t)}^t f(t-s)\lambda(s)dsdt$}}\\
          &=\lambda(t)dt
            -F^c(\varpi(t))d\fm{E}(t-\varpi(t))
            -F_t(\varpi(t)) dt,
\end{align*}
where $F_t$ is given in \eqref{eq:def-Ft-f-lambda-t-s}.
Combining the above with \eqref{eq:lemma3-at-notational-simplification} and \eqref{eq:def-Ht}, it is easy to verify 
\begin{align}
    \label{eq:lemma-At-omegat}
    d{a}(t)=F^c(\varpi(t))d\fm{E}(t-\varpi(t)).
\end{align}
To arrive at the result of this lemma, our first step is to show that
\begin{align}
  \label{eq:Etw-non-decreasing}
  \fm{E}(t-\varpi(t))
  \text{ is non-decreasing.}
\end{align}
Let $S_F=\inf\{x\geq 0: F(x)=1\}$. According the value of $S_F$ we consider the following two cases.\\
\textbf{Case~1}: $S_F=\infty$. 
Since $\varpi(t)<\infty$ on any finite time interval by \eqref{eq:queue-upbound-lemma} and \eqref{eq:qt-definition-lemma3}, one can see from \eqref{eq:lemma-At-omegat} that
\begin{align}
    \label{eq:lemma-case1-t-omegat-equal}
    \fm{E}(t-\varpi(t))-\fm{E}(0-\varpi(0))=\int_0^t \frac{1}{F^c(\varpi(s))}d{a}(s).
\end{align}
Due to the fact that $a(\cdot)$ is non-decreasing from Lemma~\ref{lem:at-nonnegative}, the above immediately yields that $\fm{E}(t-\varpi(t))$ is non-decreasing.\\
\textbf{Case~2}: $S_F<\infty$.  In this case, it is possible that $\varpi(\cdot)=S_F$ within a finite time. So \eqref{eq:lemma-case1-t-omegat-equal} may not hold. Therefore, we choose any $0\le t_1<t_2$ and consider the following two situations.\\
If $\varpi(t_1)=S_F$, then
\begin{align}
    \label{eq:lemma3-cases1-difference}
    t_2-\varpi(t_2)
    -(t_1-\varpi(t_1))=t_2-t_1+S_F-\varpi(t_2)
    \geq 0,
\end{align}
where the inequality holds due to the fact that $\varpi(t)=F_{d,t}^{-1}(q(t))\leq S_F$ for all $t\ge 0$ following from \eqref{eq:def-Fdt}.
Thus, the above inequality \eqref{eq:lemma3-cases1-difference} and \eqref{def:Et-fluid} yield $\fm{E}(t_2-\varpi(t_2))\ge \fm{E}(t_1-\varpi(t_1))$.  
\\
If $\varpi(t_1)<S_F$. Let $\tau=\inf\{s\geq t_1:\varpi(s)\geq S_F\}$ be the first time that $\varpi(t)$ is larger than or equal to $S_F$. Once $\tau=\infty$, it becomes the same issue as  Case~1. So we just need to consider $\tau<\infty$. Similar to \eqref{eq:lemma-case1-t-omegat-equal} we have
\begin{align*}
  \fm{E}(t-\varpi(t))-\fm{E}(t_1-\varpi(t_1))
  =\int_{t_1}^{t} \frac{1}{F^c(\varpi(s))}d{a}(s)
  \geq 0
  \quad\text{for all }t\in[t_1,\tau],
\end{align*}
where the last inequality holds due to the fact that $F^c(\varpi(s))>0$ on the interval $(t_1,t)$ and $a(t)$ is non-decreasing proved in Lemma~\ref{lem:at-nonnegative}. If $t_2\in(t_1,\tau]$, the above yields that $\fm{E}(t_2-\varpi(t_2))\ge \fm{E}(t_1-\varpi(t_1))$.   
If $t_2\in(\tau,\infty)$, then similar to \eqref{eq:lemma3-cases1-difference} we can apply the situation $\varpi(\tau)=S_F$  to obtain $t_2-\varpi(t_2)\geq \tau-\varpi(\tau)$. This implies that $\fm{E}(t_2-\varpi(t_2))\geq \fm{E}(\tau-\varpi(\tau))$ by \eqref{def:Et-fluid}. This together with the above inequality yields $\fm{E}(t_2-\varpi(t_2))\ge \fm{E}(t_1-\varpi(t_1))$. From the above analysis we can conclude that \eqref{eq:Etw-non-decreasing} holds in any case.

With the help of \eqref{eq:Etw-non-decreasing}, we prove \eqref{eq:lemma3-arrival-time-non-decreasing} by contradiction and assume to the contrary that there exist $0\leq \tau<t$ such that $t-\varpi(t)<\tau-\varpi(\tau)$. This implies 
\begin{align*}
  \fm{E}(\tau-\varpi{(\tau)})
  -\fm{E}(t-\varpi{(t)})
  =\int_{t-\varpi(t)}^{\tau-\varpi(\tau)}
  \lambda(s)ds\geq 0,
\end{align*}
where the equation comes from \eqref{def:Et-fluid} and the inequality follows since $\lambda(\cdot)\geq 0$. On the other hand, one can see from \eqref{eq:Etw-non-decreasing}
that $\fm{E}(t-\varpi{(t)})\ge\fm{E}(\tau-\varpi{(\tau)})$ since $t>\tau$. Therefore there must be $\fm{E}(t-\varpi(t))
  =\fm{E}(\tau-\varpi(\tau))$. This together with \eqref{eq:qt-definition-lemma3} yields
\begin{align*}
  q(t)&=\int_{t-\varpi(t)}^{t}
  F^c(t-s)\lambda(s)ds
  =\int_{\tau-\varpi(\tau)}^{t}
  F^c(t-s)\lambda(s)ds\\
  &=\int_0^{t-\tau+\varpi(\tau)}
  F^c(s)\lambda(t-s)ds,
\end{align*}
where the last equation follows by applying change of variable. By \eqref{eq:def-Fdt} and \eqref{eq:qt-definition-lemma3} we have
\begin{align*}
  q(t)=\int_0^{\varpi(t)}
  F^c(s)\lambda(t-s)ds
  =F_{d,t}(\varpi(t)).
\end{align*}
Recall the definition of $F_{d,t}^{-1}$ below \eqref{eq:def-Ht}. We can see from the above two equations that
\begin{align*}
  \varpi(t)\leq t-\tau+\varpi(\tau).
\end{align*}
The above just means $t-\varpi(t)
  \geq \tau-\varpi(\tau)$. This contradicts  the assumption. Thus we must have \eqref{eq:lemma3-arrival-time-non-decreasing} to be non-decreasing. \end{proof}   

\begin{lemma}
  \label{le:conse-tau}
  If the existence and uniqueness of the solution to \eqref{eq:key-equation-time-varying} hold on $[0,\tau]$ for some $\tau> 0$, then there exists a number $b>0$ such that the unique solution can be extended to $[0,\tau+b]$.
\end{lemma}
\begin{proof}
To prove this lemma, we analyze the following two cases.

\noindent{\textbf{Case 1}}: $\fm{X}(\tau)\leq 1$. It follows from Proposition~\ref{lem:time-shift-fluid-model}
that we can obtain the same shifted key equation as \eqref{eq:shifted-key-equation-tau}.
Thus, it is enough to prove the existence and uniqueness of the solution to
\eqref{eq:shifted-key-equation-tau} on $[0,b]$. 
Deduce from Lemma~\ref{lem:t-omegat-nondecreasing} that for all $t\geq 0$,
\begin{align*}
    \tau+t-F_{d,\tau+t}^{-1}((\fm{X}_\tau(t)-1)^+)
    \geq \tau-F_{d,\tau}^{-1}((\fm{X}_\tau(0)-1)^+)
    =\tau.
\end{align*}
Combining the above with \eqref{eq:def-Fdt} yields
\begin{align}
    \label{eq:pro_queue_upperbound}
    (\fm{X}_\tau(t)-1)^+
    \leq \int_0^t F^c(s)
    \lambda(\tau+t-s)ds
    =F_{d,\tau+t}(t).
\end{align}
{Let
\begin{align}
  \label{eq:def-truncation}
  \tfm{H}_{\tau+t}(y)
  =
  \begin{cases}
  \lambda(\tau+t)-
  F_{\tau+t}(F_{d,\tau+t}^{-1}(y)),
  &\text{if }0\leq y < F_{d,{\tau+t}}(t),\\
  \lambda(\tau+t)
  -F_{\tau+t}(t),
  &\text{if }y\geq F_{d,\tau+t}(t),
  \end{cases}
\end{align}
which actually is a truncation of $H_{\tau+t}(y)$. This is because that one
can see from \eqref{eq:def-Ht},
\begin{align}
  \label{eq:Htaut-truncation-pro}
  \tfm{H}_{\tau+t}(y)
  =H_{\tau+t}(y)
  \quad\text{for all }
  0\leq y\leq F_{d,\tau+t}(t).
\end{align}
Thus, deduce from \eqref{eq:pro_queue_upperbound} that any function $\fm{X}_\tau(t)$
satisfying \eqref{eq:shifted-key-equation-tau}  also satisfies the following
convolution equation,
\begin{align}
\label{eq:restraction-shifted-key-equation-tau}
  \begin{split}
  \fm{X}_\tau(t)&=\fmm{Z}(\tau)(C_t)+\fm{Q}(\tau)G^c(t)\\
  &\quad+\frac{1}{\mu}\int_0^t \tfm{H}_{\tau+t-s}((\fm{X}_\tau(t-s)-1)^+)dG_e(s)
  +\int_0^t (\fm{X}_{\tau}(t-s)-1)^+dG(s).
  \end{split}
\end{align}
Regarding $\tau$ as a starting point, the above equation thus becomes a key
equation of a fluid model with initial state $(\fmm{R}_\tau(0),\fmm{Z}_\tau(0))$
satisfying \eqref{def:fmmRt-Zhang-shift}--\eqref{def:fmmZt-Zhang-shift} at
time $t=0$ and external arrival rate being $\lambda_\tau(t):=\lambda(\tau+t)$.
Thus for any function $\fm{X}_\tau(t)$ satisfying \eqref{eq:restraction-shifted-key-equation-tau}
we can obtain similar results as Lemmas~\ref{lem:at-nonnegative}--\ref{lem:t-omegat-nondecreasing}
using the same argument (with different initial states and the external arrival
processes). Especially, replacing $\omega(0)$ and $\lambda(t)$ in \eqref{eq:queue-upbound-lemma}
respectively with $\omega_\tau(0)$ and $\lambda_\tau(t)$, we can obtain the
following inequality for any solution $\fm{X}_\tau(t)$ satisfying \eqref{eq:restraction-shifted-key-equation-tau},
\begin{align*}
  (\fm{X}_\tau(t)-1)^+
  \leq 
  \int_0^{t+\omega_\tau(0)}F^c(s)
  \lambda_\tau(t-s)ds.
\end{align*}
The proof is essentially the same as Lemma~\ref{lem:queue-upbound-lemma}, so we
omit it for brevity. Since $\omega_\tau(0)=0$ due to the fact  $\fm{X}_\tau(0)\leq
1$ and it satisfies \eqref{def:fmmRt-Zhang-shift} at $t=0$ that the right-hand
side of the above inequality equals  $F_{d,\tau+t}(t)$ by \eqref{eq:def-Fdt}.
This together with \eqref{eq:pro_queue_upperbound} and \eqref{eq:Htaut-truncation-pro}
immediately yields that the convolution equations \eqref{eq:shifted-key-equation-tau}
and \eqref{eq:restraction-shifted-key-equation-tau} have same solution $\fm{X}_\tau(t)$
(if any) for all $t\geq 0$. Thus instead of analyzing \eqref{eq:shifted-key-equation-tau}
we just need to  prove the existence and uniqueness of the solution to
\eqref{eq:restraction-shifted-key-equation-tau} on $[0,b]$. 
Let $M$ be any strictly positive number and $S_F=\inf\{x\geq 0: F(x)=1\}$.
 From \eqref{eq:def-truncation} the following derivative is bounded for all
$t\in[0,\frac{S_F\wedge M}{2}]$:
\begin{align*}
  \frac{d}{dy}\tfm{H}_{\tau+t}(y)
  = -\frac{f(F_{d,\tau+t}^{-1}(y))}{F^c(F_{d,\tau+t}^{-1}(y))}
  \geq -\frac{L_F}{F^c(t)}
  \geq - \frac{L_F}{F^c(\frac{S_F\wedge M}{2})},
  \quad \text{if }0\leq y \le F_{d,{\tau+t}}(t),
\end{align*}
where $L_F$ is denoted to be the Lipschitz constant of $F$ by Assumption~\ref{assump:GF}.
We can pick $b_1=\frac{S_F\wedge M}{2}$ and then for any $t\in[0,b_1]$ the
function $\tfm{H}_{\tau+t}(\cdot)$ in \eqref{eq:restraction-shifted-key-equation-tau}
is Lipschitz continuous. Let $L=\frac{L_F}{F^c(\frac{S_F\wedge
M}{2})}$ be the Lipschitz constant.}
By Assumption~\ref{assump:GF}, there exists a $b_2>0$ such that
\begin{align*}
  \kappa := 
  \frac{1}{\mu}L[G_e(b_2)-G_e(0)]+[G(b_2)-G(0)]<1.
\end{align*}
Let $b=\min\{b_1,b_2\}$. 
For any $x\in{C}[0,b]$, define 
\begin{align*}
  \Psi(x)(t) &= \fmm{Z}(\tau)(C_t)+\fm{Q}(\tau)G^c(t)\\
             &\quad  +\frac{1}{\mu}\int_0^t \text{{$\tfm{H}_{\tau+t-s}$}}((x(t-s)-1)^+)dG_e(s)
               +\int_0^t(x(t-s)-1)^+dG(s).
\end{align*}
It is clear that $\Psi(x)(t)$ is continuous in $t$, so $\Psi$ is a mapping
from $\C[0,b]$ to $\C[0,b]$.
Let $\rho(x,x')=\sup_{t\in[0,b]}|x(t)-x'(t)|$ denote the uniform distance
between two functions in $\C[0,b]$.
For any $x,x'\in C[0,b]$, we have
\begin{align*}
    \rho(\Psi(x),\Psi(x'))
    &\leq \sup_{t\in[0,b]} \frac{1}{\mu}\int_0^t
          L|(x(t-s)-1)^+-(x'(t-s)-1)^+|dG_e(s)\\
    &\quad + \sup_{t\in[0,b]}\int_0^t
           |(x(t-s)-1)^+-(x'(t-s)-1)^+| dG(s)\\
    &\leq \frac{1}{\mu} L\int_0^{b} \rho(x,x')dG_e(s)
          +\int_0^{b}\rho(x,x')dG(s)\\
    &\leq \kappa \rho(x,x').
\end{align*}
Since $\kappa<1$, $\Psi$ is a contraction mapping on $\C[0,b]$ under the
uniform topology $\rho$.
Note that $\C[0,b]$ is complete under the uniform topology of $\rho$ (cf.~p.~80
in \cite{Billingsley1999}).
Thus, by the contraction mapping theorem (e.g., Theorem~3.2 in \cite{HunterNach2001}),
$\Psi$ has a unique fixed point $x$, i.e., $x=\Psi(x)$.
This proves {that \eqref{eq:restraction-shifted-key-equation-tau}} has a
unique solution on $[0,b]$. {It is also the unique solution to \eqref{eq:shifted-key-equation-tau}
on $[0,b]$ as argued in the above.
}

\noindent{\textbf{Case~2}}: $\fm{X}(\tau)>1$. As in Case~1, we also have the shifted key equation \eqref{eq:shifted-key-equation-tau}. Due to the
continuity there exists $b_3>0$ such that
\begin{align}
    \label{eq:proof-Xt-1}
    \fm{X}_\tau(t)\geq  1 \quad\text{for all }t\in[0,b_3].
\end{align}
For notational simplicity, denote $q_\tau(t)=(\fm{X}_\tau(t)-1)^+$ and
\begin{align}
    \label{eq:case2-proof-non-decreasing}
    a_\tau(t)=\int_0^t H_{\tau+s}(q_\tau(s))ds-q_\tau(t)+q_\tau(0).
\end{align}
By \eqref{eq:Geequilibrium} and \eqref{eq:shifted-key-equation-tau},
\begin{align*}
  \fm{X}_\tau(t)&=\fmm{Z}_\tau(0)(C_t)+q_{\tau}(0)G^c(t)
  +\int_0^t H_{\tau+t-s}(q_\tau(t-s))[1-G(s)]ds
  +\int_0^t q_{\tau}(t-s)dG(s)\\
  &=\fmm{Z}_\tau(0)(C_t)+q_{\tau}(0)G^c(t)
  +\int_0^t H_{\tau+s}(q_\tau(s))ds
  -\int_0^t H_{\tau+s}(q_\tau(s))G(t-s)ds
  +\int_0^t q_{\tau}(t-s)dG(s).
\end{align*}
The second last term on the above equation satisfies
\begin{align*}
  \int_0^t H_{\tau+s}(q_\tau(s))G(t-s)ds
  &=\int_0^t \int_0^{t-s}
   H_{\tau+s}(q_\tau(s))dG(x)ds\\
  &=
  \int_0^t\int_0^{t-x} H_{\tau+s}(q_\tau(s))ds dG(x),
\end{align*}
where the last equality follows by changing the order of integration. So we obtain
\begin{align*}
  \fm{X}_\tau(t)
  =\fmm{Z}_\tau(0)(C_t)+q_{\tau}(0)
  +\int_0^t H_{\tau+s}(q_\tau(s))ds
  -\int_0^t\left[\int_0^{t-x} H_{\tau+s}(q_\tau(s))ds
  -q_\tau (t-x)+q_\tau(0)\right] dG(x).
\end{align*}
According to the above definition of $a_{\tau}(t)$, we have
\begin{align*}
  a_\tau(t)=\fm{X}_\tau(t)-q_{\tau}(t)-\fmm{Z}_\tau(0)(C_t)
  +\int_0^t a_{\tau}(t-s)dG(s).
\end{align*}
By \eqref{eq:proof-Xt-1},
the above becomes
\begin{align*}
    {a}_\tau(t)=1-\fmm{Z}_\tau(0)(C_t)+\int_0^t {a}_\tau(t-s)dG(s),
    \quad t\in[0,b_3].
\end{align*}
Let $G^{n*}$ be the $n$-fold convolution of $G$ with itself, and denote $U_G(t)=\sum_{i=0}^\infty
G^{n*}$. The solution to the above renewal equation is
\begin{align*}
  a_\tau(t)
  = \int_0^t {(}
  1-\fmm{Z}_\tau(0)(C_{t-s})
  {)}
  dU_G(s),
  \quad t\in[0,b_3].
\end{align*}
It is clear that $a_\tau(t)$ is continuous.
Since $H_{\tau+t}(\cdot)$ is continuous, with a known ${a}_\tau(t)$ there
exists a  continuous solution $q_\tau(t)$ to the equation \eqref{eq:case2-proof-non-decreasing}
following from Theorem~II.1.1 in \cite{Miller1971}. 

Next, we prove the uniqueness. Assume that ${q}_1(t)$ and ${q}_2(t)$ satisfy
\eqref{eq:case2-proof-non-decreasing} on the interval $[0,b_3]$.  Let
\begin{align*}
    \mathcal{L}(t):=(q_1(t)-q_2(t))^2,
    \quad t\in[0,b_3].
\end{align*}
Then, on the interval $[0,b_3]$ we can see from \eqref{eq:case2-proof-non-decreasing}
that
\begin{align*}
    \mathcal{L}'(t)=2[q_1(t)-q_2(t)]
                     [H_{\tau+t}(q_1(t))-H_{\tau+t}(q_2(t))]
                  \leq 0,
\end{align*}
where the last inequality is due to the fact that $H_{\tau+t}(\cdot)$ is
non-increasing; see \eqref{eq:def-Ht}. Thus
$\mathcal{L}(t)$ is non-increasing on $[0,b_3]$. Since $\mathcal{L}(0)=0$
and
$\mathcal{L}(t)\geq 0$, $\mathcal{L}(t)=0$ for all $t\in[0,b_3]$.
Hence
\begin{align*}
    q_1(t)=q_2(t)\quad\text{for all }t\in[0,b_3].
\end{align*}
Thus  \eqref{eq:shifted-key-equation-tau} only has one
solution on the interval $[0,b_3]$. So we have the existence and uniqueness
of the solution to \eqref{eq:key-equation-time-varying} on the interval $[0,\tau+b_3]$.
In fact, our analysis shows that we can further extend the solution to a
point where $\fm{X}(\cdot)$ reaches $1$. Starting from there, we can apply
Case~1 to extend the solution to an extra interval with length $b$. Again, we can at least extend the unique solution of \eqref{eq:key-equation-time-varying}
to the interval $[0,\tau+b]$ in this case.
\end{proof}

\section*{Acknowledgement}
The authors are grateful to the AE and the anonymous referees for constructive comments and suggestions. The research is supported in part by the Hong Kong Research Grants Council [Grants GRF-16501015 and GRF-16201417].

\bibliography{pub}
\end{document}